\documentclass[12pt,leqno]{amsart}
\usepackage{indentfirst, amssymb,amsmath,amsthm}
\newtheorem*{theoA}{Theorem A}
\newtheorem*{theoB}{Theorem B}
\newtheorem*{theoC}{Theorem C}
\newtheorem*{theoD}{Theorem D}

\newtheorem{theo}{Theorem}[section]
\newtheorem{lem}{Lemma}[section]
\newtheorem{ques}{Question}[section]
\newtheorem{cor}{Corollary}[section]

\newtheorem{exm}{Example}[section]
\newtheorem{defi}{Definition}[section]
\newtheorem{rem}{Remark}[section]
\newcommand{\ol}{\overline}
\newcommand{\be}{\begin{equation}}
\newcommand{\ee}{\end{equation}}
\newcommand{\beas}{\begin{eqnarray*}}
\newcommand{\eeas}{\end{eqnarray*}}
\newcommand{\bea}{\begin{eqnarray}}
\newcommand{\eea}{\end{eqnarray}}

\numberwithin{equation}{section}
\begin{document}
\title[ UNIQUENESS OF A MEROMORPHIC FUNCTIONS]{Uniqueness of a meromorphic function and its linear difference polynomial sharing two sets with finite weights}
\date{}
\author[ G. Haldar ]{ Goutam Haldar}
\date{}
\address{ Department of Mathematics, Malda College, Rabindra Avenue, Malda, West Bengal 732101, India.}
\email{goutamiitm@gmail.com}
\maketitle
\let\thefootnote\relax
\footnotetext{Type set by \AmS -\LaTeX}
\setcounter{footnote}{0}
\begin{abstract}
In this paper, we investigate the uniqueness property of meromorphic functions together with its linear difference polynomial sharing two sets. Using the polynomial introduced in [FILOMAT 33(18)(2019), 6055-6072], we have improved the result of Li-Chen [Abstract and Applied Analysis, 2014,Article ID 894968] in sense of reducing cardinalities of the main set $ S $ and the associated weights. Some examples have been exhibited to validate our certain claims in the main result.
\end{abstract}
\section{Introduction, Definitions and Results}
Let $f$ and $g$ be two non-constant meromorphic functions defined on the set of complex numbers $\mathbb{C}$ and  $a\in\mathbb{C}\cup\{\infty\}$. We say that $f$ and $g$ share the value $a$ CM (counting multiplicities) if $f-a$ and $g-a$ have the same set of zeros with the same multiplicities and if we do not the multiplicities, then $f$ and $g$ are said to share the value $a$ IM (ignoring multiplicities).\par
Throughout the paper, we have used the standard notations and definitions of value distribution theory of meromorphic functions introduced in \cite{17}.
We recall that $T(r,f)$ denotes the Nevanlinna characteristic function of the non-constant meromorphic function. Also we denote by $S(r,f)$ any quantity satisfying $T(r,f)=o(T(r,f))$ as $r\rightarrow \infty$ possibly outside a finite set of logarithmic measure and $N(r,a;f)$ ($\ol N(r,a;f)$) denotes the counting function (reduced counting function) of $a$-points of meromorphic functions $f$. A meromorphic function $a(z)$ is said to be a small function of $f$ if $T(r,a)=o(T(r,f))$. Let $S(f)$ be the set of all small functions of $f$.
For a set $S\subset S(f)$, we define
\beas E_{f}(S)=\bigcup_{a\in S}\{z|f(z)-a(z)=0\},\eeas where each zero is counted according to its multiplicity and
$$\ol E_{f}(S)=\bigcup_{a\in S}\{z|f(z)-a(z)=0 \},\;\text{where each zero is counted only once}.$$\par
If $E_{f}(S)=E_{g}(S)$, we say that $f$, $g$ share the set S CM and if $\ol E_{f}(S)=\ol E_{g}(S)$, we say $f$, $g$ share the set S IM.\par
In 2001, Lahiri  (\cite{13}, \cite{14}) introduced a remarkable notion called weighted sharing of values and sets  which which renders a useful tool in the literature. We explain the notion in the following.
\begin{defi}\cite{13}
Let $k$ be a non-negative integer or infinity. For $a\in \mathbb{C}\cup\{\infty\}$ we denote by $E_{k}(a,f)$ the set of all $a$-points of $f$, where an $a$ point of multiplicity $m$ is counted $m$ times if $m\leq k$ and $k+1$ times if $m>k.$ If $E_{k}(a,f)=E_{k}(a,g),$ we say that $f$, $g$ share the value $a$ with weight $k$.
\end{defi}
We write $f$, $g$ share $(a,k)$ to mean that $f,$ $g$ share the value $a$ with weight $k.$ Clearly if $f,$ $g$ share $(a,k)$ then $f,$ $g$ share $(a,p)$ for any integer $p$, $0\leq p<k.$ Also we note that $f,$ $g$ share a value $a$ IM or CM if and only if $f,$ $g$ share $(a,0)$ or $(a,\infty)$ respectively.\par
\begin{defi}\cite{13}
Let $S$ be a set of distinct elements of $\mathbb{C}\cup\{\infty\}$ and $k$ be a non-negative integer or $\infty$. We denote by $E_{f}(S,k)$ the set $\bigcup_{a\in S}E_{k}(a,f)$.\par Clearly $E_{f}(S)=E_{f}(S,\infty)$ and $\ol E_{f}(S)=E_{f}(S,0)$.
\end{defi}
Suppose $p$ be a non-zero complex constant. We define the shift of $f(z)$ by $f(z+p)$ and define the difference operators by $$\Delta_{p}f(z)=f(z+p)-f(z),$$ $$\Delta_{p}^{n}f(z)=\Delta_{p}^{n-1}(\Delta_{p}f(z)),\, n\in \mathbb{N},\, n\geq 2.$$
In 1995, Li-Yang in \cite{18} obtained the following result.

\begin {theoA}\cite{18} Let $m\geq 2$ and let $n> 2m+6$ with $n$ and $n-m$ having no common factors. Let $a$ and $b$ be two non-zero constants such that the equation $\omega^{n}+a\omega^{n-m}+b=0$ has no multiple roots. Let $S=\{\omega\mid\omega^{n}+a\omega^{n-m}+b=0\}$. Then for any two non constant meromorphic functions $f$ and $g$, the conditions $E_{f}(S,\infty)=E_{g}(S,\infty)$ and $E_{f}(\{\infty\},\infty)=E_{g}(\{\infty\},\infty)$ imply $f\equiv g$.
\end {theoA}

Let us explain some standard definitions and notations of the value distribution theory available in \cite{17} which will be used in the paper.
\begin{defi}\cite{12a}
For $a\in \mathbb{C}\cup \{\infty\},$ we denote by $N(r,a;f\mid=1)$ the counting function of simple $a$-points of $f.$ For a positive integer $m,$ we denote by $N(r,a;f\mid\leq m)$ $(N(r,a;f\mid\geq m))$ the counting function of those $a$-point of $f$ whose multiplicities are not greater (less) than $m$, where each $a$-point is counted according to its multiplicity.
\end{defi}
$\ol N(r,a;f\mid\leq m)$ $(\ol N(r,a;f\mid\geq m))$ are defined similarly except that in counting the $a$-points of $f$ we ignore the multiplicity. Also $N(r,a;f\mid< m)$, $N(r,a;f\mid> m),$ $\ol N(r,a;f\mid< m)$ and $\ol N(r,a;f\mid>m)$ are defined similarly.
\begin{defi}\cite{14}
For $a\in \mathbb{C}\cup \{\infty\},$ we denote by $N_{2}(r,a;f)=\ol N(r,a;f)+\ol N(r,a;f\mid\geq 2).$
\end{defi}
\begin{defi}\cite{14}
Let $f$ and $g$ share a value $a$ IM. We denote by $\ol N_{*}(r,a;f,g)$ the counting function of those $a$-points of $f$ whose multiplicities differ from the multiplicities of the corresponding $a$-points of $g$.
\end{defi}
Recently, the shift and difference analogue of the Navanlinna theory has been established (see, e.g. \cite{23,Hal & Kor & JMAA & 2006,Hei & Kor & Lai & Rei & JMMA & 2009}). Many researchers started to consider the uniqueness of meromorphic functions sharing values with their shifts or differences operators (see, e.g. \cite{Aha & RM & 2019,Aha & BB & 2019,Ban & Aha & MS & 2019,Ban & Aha & JCMA & 2020}.) In 2010, Zhang \cite{16} considered a meromorphic function $f(z)$ sharing sets with its shift $f(z+p)$ and obtained the following result.

\begin{theoB}\cite{16}
Let $m\geq 2$ and $n\geq 2m+4$ with $n$ and $n-m$ having no common factors. Let $S=\{\omega \mid\omega^{n}+a\omega^{n-m}+b=0\}$, where $a$ and $b$ be two non-zero complex constants such that the equation $\omega^{n}+a\omega^{n-m}+b=0$ has  no multiple roots. Suppose that $f(z)$ is a non-constant meromorphic function of finite order. Then $E_{f(z)}(S,\infty)=E_{f(z+p)}(S,\infty)$ and $E_{f(z)}(\{\infty\},\infty)=E_{f(z+p)}(\{\infty\},\infty)$ imply $f(z)\equiv f(z+p).$
\end{theoB}
For an analogue result in difference operator, Chen-Chen  \cite{29} proved the following result.
\begin{theoC}\cite{29}
Let $m\geq 2$ and $n\geq 2m+4$ with $n$ and $n-m$ having no common factors. Let $a$ and $b$ be two non-zero complex constants such that the equation $\omega^{n}+a\omega^{n-m}+b=0$ has  no multiple roots. Let $S=\{\omega \mid\omega^{n}+a\omega^{n-m}+b=0\}$. Suppose that $f(z)$ is a non-constant meromorphic function of finite order satisfying $E_{f}(S,\infty)=E_{\Delta_{p}f}(S,\infty)$ and $E_{f}(\{\infty\},\infty)=E_{\Delta_{p}f}(\{\infty\},\infty)$. If \beas N(r,0;\Delta_{p}f)=T(r,f)+S(r,f),\eeas then $f\equiv \Delta_{p}f.$
\end{theoC}

In $2014$, Li-Chen \cite{18a} considered a linear difference polynomial of $f$ in the following manner \be\label{e1.1a} L(z,f)=b_{k}(z)f(z+c_{k})+\ldots+b_{0}(z)f(z+c_{0}),\ee where $b_{k}(z)(\not\equiv 0)$, $\ldots,b_{0}(z)$ are small functions of $f$, $c_{0},c_{1},\ldots,c_{k}$ are complex constants and $k$ is a non-negative integer satisfying one of the following conditions: \bea\label{e1.1b} \;b_{0}(z)+\ldots+b_{k}(z)\equiv 1,\\\label{e1.1c}\;b_{0}(z)+\ldots+b_{k}(z)\equiv 0\eea and obtained the following theorem.
\begin{theoD}\cite{18a}
Let $m\geq 2$ and $n\geq 2m+4$ with $n$ and $n-m$ having no common factors. Let $a$ and $b$ be two non-zero complex constants such that the equation $\omega^{n}+a\omega^{n-m}+b=0$ has  no multiple roots. Let $S=\{\omega \mid\omega^{n}+a\omega^{n-m}+b=0\}$. Suppose that $f(z)$ is a non-constant meromorphic function of finite order and $L(z,f)$ is of the form (\ref{e1.1a}) satisfying the conditions (\ref{e1.1b}) and (\ref{e1.1c}). If  $E_{f}(S,\infty)=E_{L(z,f)}(S,\infty)$, $E_{f}(\{\infty\},\infty)=E_{L(z,f)}(\{\infty\},\infty)$ and $N(r,0;L(z,f))=T(r,f)+S(r,f)$, then $$f(z)\equiv L(z,f).$$
\end{theoD}
\begin{rem}\label{r1.1}
In $(\ref{e1.1a})$, if we assume $c_{j}=jp$, $j=0,1,\ldots,k$ and $b_{k}(z)={k\choose 0}$, $b_{k-1}=-{k\choose 1}$, $b_{k-2}={k\choose 2}, \ldots,b_{0}=(-1)^{k}{k\choose k}$, then it can be easily seen that $L(z,f)=\Delta^{k}_{p}f$.

\end{rem}

\begin{rem}
From the above discussions, it is to be observed that in Theorem B, Theorem C and Theorem D the minimum cardinality of the main range set is 9. Also they have got their results under CM sharing hypothesis.
\end{rem}

So, naturally one may ask the following questions.\par
\begin{ques}
	Can we further get some other main range set whose cardinalities is less than $9\;?$
\end{ques}
\begin{ques}
	Is it possible to get the uniqueness of the meromorphic function $f$ with its linear difference polynomials under relax sharing hypothesis without assuming the conditions $(\ref{e1.1b})$, $(\ref{e1.1c})$ and $ N(r,0;L(z,f))=T(r,f)+S(r,f)$ $?$
\end{ques}
The above two questions are the motivation of writing the paper. In this paper, using a new type of range set by taking the zeros of the polynomial introduced in \cite{2*}, we proved a result which improves theorem D in some sense without assuming the conditions (\ref{e1.1b}), (\ref{e1.1c}) and $ N(r,0;L(z,f))=T(r,f)+S(r,f)$.\par

We now recall here the uniqueness polynomial $ P(z) $ introduced by Banerjee-Ahamed  \cite{2*} \be\label{e1.1} P(z)=az^{n}+bz^{2m}+cz^{m}+d,\ee where $n, m$ are positive integers with $n>2m$, $\text{gcd}(n,m)=1$ and $a,b,c,d\in\mathbb{C}-$$\{0\}$, $\displaystyle\frac{c^{2}}{4bd}=\displaystyle\frac{n(n-2m)}{(n-m)^{2}}\neq 1$; $a\neq \gamma_{j}=-\displaystyle\frac{1}{n}(2bm\omega_{j}^{2m}+cm\omega_{j}^{m})$, with $\omega_{j}$ be the roots of the equation $z^{m}+\displaystyle\frac{2nd}{c(n-m)}=0$.\par Then by simple calculation it can easily seen that all the zeros of $P(z)$ are simple.

Let us define \be\label{e1.1g} W(z)=-\displaystyle\frac{az^{n}}{bz^{2m}+cz^{m}+d}.\ee
Now, \beas W(z)-\displaystyle\frac{a}{\gamma_{j}}&=&-a\displaystyle\frac{\gamma_{j}z^{n}+bz^{2m}+cz^{m}+d}{\gamma_{j}(bz^{2m}+cz^{m}+d)}\\&=&-a\displaystyle
\frac{Q(z)}{\gamma_{j}(bz^{2m}+cz^{m}+d)},\eeas where $Q(z)=\gamma_{j}z^{n}+bz^{2m}+cz^{m}+d$.\par
It is to be observed that $Q(\omega_{j})=Q^{'}(\omega_{j})=Q^{''}(\omega_{j})=0$ but $Q^{'''}(\omega_{j})\neq 0$ for $j=1,2,\ldots,m$.\par
Therefore, \be\label{e1.1f} W(z)-\beta_{j}=-a\displaystyle\frac{\displaystyle\prod_{j=1}^{m}(z-\omega_{j})^{3}R_{n-3m}(z)}{bz^{2m}+cz^{m}+d},\ee where $R_{n-3m}(z)$ is a polynomial of degree $n-3m$ and \be\label{e1.1g} \beta_{j}=\displaystyle\frac{a}{\gamma_{j}},\;j=1,2,\ldots,m.\ee
 The following theorem is the main result of the paper.\par
\begin{theo}\label{t1}
Let $S=\{z\mid P(z)=0\},$ where $P(z)$ is a polynomial given by $(\ref{e1.1})$ and $n(\geq2m+3),$ $m(\geq1)$ be two positive integers such that $gcd(n,m)=1$, $a,b,c,d$ are non zero complex numbers, $\displaystyle\frac{c^{2}}{4bd}=\displaystyle\frac{n(n-2m)}{(n-m)^{2}}\neq 1$ and $a\neq \gamma_{j}$ for $j=1,2,\ldots,m$. Let $f(z)$ be a transcendental meromorphic function of finite order. Suppose $E_{f(z)}(S,3)=E_{L(z,f)}(S,3)$ and $E_{f(z)}(\{\infty\},0)=E_{L(z,f)}(\{\infty\},0)$, where $L(z,f)$ is defined in (\ref{e1.1a}) . Then $$f(z)\equiv L(z,f).$$

\end{theo}
Keeping in view of the {{Remark $\ref{r1.1}$}}, we can easily obtain the following corollary from the above theorem.
\begin{cor}\label{c1}
Suppose $S$ be defined as in the theorem \ref{c1}, $n(\geq2m+3),$ $m(\geq1)$ be two positive integers such that $gcd(n,m)=1$, $a,b,c,d$ are non zero complex numbers, $\displaystyle\frac{c^{2}}{4bd}=\displaystyle\frac{n(n-2m)}{(n-m)^{2}}\neq 1$ and $a\neq \gamma_{j}$ for $j=1,2,\ldots,m$. Let $f(z)$ be a transcendental meromorphic function of finite order and $p$ be a non-zero complex constant. Suppose $E_{f(z)}(S,3)=E_{L(z,f)}(S,3)$ and $E_{f(z)}(\{\infty\},0)=E_{L(z,f)}(\{\infty\},0)$, then $f(z)\equiv \Delta^{k}_{p}f.$
\end{cor}
\begin{rem}
    Next example shows in the main result, the polynomial $ P(z) $ can not be choose arbitrarily.
\end{rem}
\begin{exm} Since, minimum degree of the polynomial $ P(z) $ is $ 5 $, so let us consider $ P(z)=z^5-1 $. So, $ S=\{:P(z)=0\}=\{1,\theta,\theta^2,\theta^3,\theta^5 \}, $ where $ \theta=\displaystyle\cos\left(\frac{2\pi}{5}\right)+i\sin\left(\frac{2\pi}{5}\right). $ Let \beas f(z)=\left(1+\sqrt[k]{\theta}\right)^{z/c}\frac{\sin\left(\frac{2\pi z}{c}\right)}{e^{2\pi iz/c}-1}. \eeas Verify that $ \Delta^{k}_{c}f=\theta f $, and hence $ \Delta^{k}_{c}f $ and   $E_{f(z)}(S,3)=E_{\Delta^{k}_{c}f}(S,3)$ and $E_{f(z)}(\{\infty\},0)=E_{\Delta^{k}_{c}f}(\{\infty\},0)$, but $ f\not\equiv \Delta^{k}_{c}f. $
\end{exm}
\begin{rem}
    From the next example we see that, it is not necessary $ f $ to be of finite order in the main result.
\end{rem}
\begin{exm}
    Consider the function \beas f(z)=2^{z/c}{\frac{e^{\sin\left(\frac{2\pi z}{c}\right)}}{\sin\left(\frac{2\pi z}{c}\right)-1}}.\eeas Clearly, we see that order of $ f $ is infinite, and $E_{f(z)}(S,3)=E_{\Delta^{k}_{c}f}(S,3)$ and $E_{f(z)}(\{\infty\},0)=E_{\Delta^{k}_{c}f}(\{\infty\},0)$, also $ f\equiv \Delta^{k}_{c}f. $
\end{exm}
\begin{rem}
The following examples satisfy the conditions of theorem \ref{t1} and note that, it is not necessary to assume the  conditions (\ref{e1.1b}) or (\ref{e1.1c}) in order to get the uniqueness of $f(z)$ with $L(z,f)$.\par
\end{rem}
\begin{exm}\label{1.1}
Let $f(z)=e^{\frac{\log3}{p}z}$, where $p$ is a non-zero complex number. Let $L(z,f)=f(z+p)-2f(z)$. Clearly we have $L(z,f)=f(z)$.
\end{exm}
\begin{exm}\label{1.2}
Let $f(z)=e^{\frac{i\pi}{p}z}$, where $c$ is a non-zero complex number. Suppose $L(z,f)=-f(z-2p)+f(z-p)+3f(z)$. Then, clearly $L(z,f)=f(z)$.
\end{exm}
\begin{exm}\label{1.3}
Let $f(z)=e^{z}$ and $c_{1}$, $c_{2}$ and $c_{3}$ are three complex numbers. Suppose $L(z,f)=f(z+c_{3})+f(z+c_{2})+f(z+c_{1})-f(z)$. Then $L(z,f)=(-1+e^{c_{1}}+e^{c_{2}}+e^{c_{3}})f(z)$. If we choose $c_{1}$, $c_{2}$ and $c_{3}$ in such a way that $e^{c_{1}}+e^{c_{2}}+e^{c_{3}}=2$, then $L(z,f)=f(z)$.\end{exm}
\begin{exm}\label{1.4}
Let $f(z)=\frac{e^{z}}{1-\cos(\frac{i\pi}{p}z)}$, where $p$ is a non-zero complex number. Suppose $L(z,f)=2f(z+2ip)+f(z-2ip)-f(z)$. Then $L(z,f)=(-1+2e^{2ip}+e^{-2ip})f(z)$. Now let us choose $p\in \mathbb{C}$ such that $p=n\pi$, where $n=0,\pm 1,\pm 2,\ldots$. Then $L(z,f)=f(z)$.
\end{exm}

\section{Lemmas} In this section, we present some lemmas which will be needed in the sequel. Let us define three functions $F$, $G$ in $\mathbb{C}$ by  \be\label{e2.1} F=W(f)=-\frac{af^{n}}{bf^{2m}+cf^{m}+d}\ee and \be\label{e2.2}  G=W(L(z,f))=-\frac{aL(z,f)^{n}}{bL(z,f)^{2m}+cL(z,f)^{m}+d}.\ee We also  denote by $H$, $V$ the following functions \beas H=\Big(\frac{F^{''}}{F^{'}}-\frac{2F^{'}}{F-1}\Big)-\Big(\frac{G^{''}}{G^{'}}-\frac{2G^{'}}{G-1}\Big),\eeas \beas V=\frac{F^{'}}{F(F-1)}-\frac{G^{'}}{G(G-1)}.\eeas 
\begin{lem}\label{2.1}\cite{14}
Let $F$, $G$ be two non-constant meromorphic functions such that they share $(1,1)$ and $H\not\equiv 0.$ Then \beas N(r,1;F\mid=1)=N(r,1;G\mid=1)\leq N(r,H)+S(r,F)+S(r,G).\eeas
\end{lem}
\begin{lem}\label{2.2}\cite{2}
 Let $F,$ $G$ be two non-constant meromorphic functions sharing $(1,t),$ where $1\leq t<\infty.$ Then \beas && \ol N(r,1;F) + \ol N(r,1;G)-\ol N(r,1;F \mid= 1) +(t-\frac{1}{2})\ol N_{*}(r,1;F,G) \\&\leq& \frac{1}{2} [N(r,1;F) + N(r,1;G)].\eeas
\end{lem}
\begin{lem}\label{2.3}
Suppose $F$, $G$ share $(1,0)$, $(\infty,0)$ and $\beta_{j}$ be defined as in $(\ref{e1.1f})$ are non-zero complex numbers. If $H\not \equiv 0,$ then, \beas  N(r,H)&&\leq N(r,0;F \mid\geq 2) + N(r,0;G \mid\geq 2) + \sum_{j=1}^{m}\ol N(r,\beta_{j};F \mid\geq 2) \\&&+\sum_{j=1}^{m}\ol N(r,\beta_{j};G \mid\geq 2) +\ol N_{*}(r,1;F,G) + \ol N_{*}(r,\infty;F,G) + \ol N_{0}(r,0;F^{'}) \\&&+ \ol N_{0}(r,0;G^{'})+S(r,F)+S(r,G),\eeas where $\ol N_{0}(r,0;F^{'})$ is the reduced counting function of those zeros of $F^{'}$ which are not the zeros of $F(F -1)\prod_{j=1}^{m}(F-\beta_{j})$ and $\ol N_{0}(r,0;G^{'})$ is similarly defined.
\end{lem}
\begin{proof}
 By the definition of $H$ we verify that the possible poles of $H$ occur from the following six cases: (i) The multiple zeros of $F$ and $G$. (ii) The multiple $\beta_{j}$- points of $F$ and $G$ for each $j=0,1,2,\ldots,m-1$. (iii) Those common poles of $F$ and $G,$ where each such pole of $F$ and $G$ has different multiplicities related to $F$ and $G$. (iv) Those common $1$-points of $F$ and $G,$ where each such point has different multiplicities related to $F$ and $G$. (v) The zeros of $f^{'}$ which are not zeros of $F(F-1)\prod_{j=1}^{m}(F-\beta_{j})$. (vi) The zeros of $G^{'}$ which are not zeros of $G(G-1)\prod_{j=1}^{m}(G-\beta_{j})$. Since all the poles of $H$ are simple, the lemma follows.
\end{proof}
\begin{lem}\label{2.5}\cite{8}
Let $f$ be a non-constant meromorphic function and $P(f)=a_0+a_1f+a_2f^{2}+\ldots +a_{n}f^{n},$ where $a_0,a_1,a_2,\ldots, a_{n}$ are constants and $a_{n}\neq 0$. Then $$T(r,P(f))=nT(r,f)+O(1).$$
\end{lem}
\begin{lem}\label{2.8}
Let $F,$ $G$ be given by $(\ref{e2.1})$ and $(\ref{e2.2})$, where $n(\geq 2m+3)$ be an integer and $H\not\equiv 0$.  Suppose that $F,$ $G$ share $(1,t)$ and $f$, $L(z,f)$ share $(\infty,k),$ where $2\leq t<\infty$. Then, for the complex numbers $\beta_{j}$ as given by $(\ref{e1.1g})$, we have  \beas &&n\left(m+\frac{1}{2}\right)\{T(r,f)+T(r,L(z,f))\}\\&\leq&\; 2\{\ol N(r,0;f)+\ol N(r,0;L(z,f))\}+\ol N(r,\infty;f)+\ol N(r,\infty;L(z,f))\\&&+\sum_{j=1}^{m}N_{2}(r,\beta_{j};F)+\sum_{j=1}^{m}N_{2}(r,\beta_{j};G)+\ol N_{*}(r,\infty;f,L(z,f))\\&&-\left(t-\frac{3}{2}\right)\ol N_{*}(r,1;F,G)+S(r,f)+S(r,L(z,f)).\eeas
\end{lem}
\begin{proof}
Using the second fundamental theorem of Nevalinna, we get \beas&& (m+1)\{T(r,F)+T(r,G)\}\\&\leq& \ol N(r,0;F)+\ol N(r,1;F)+\ol N(r,\infty;F)+\sum_{j=1}^{m}\ol N(r,\beta_{j};F)+\ol N(r,0;G)\\&&+\ol N(r,1;G)+\ol N(r,\infty;G)+\sum_{j=1}^{m}\ol N(r,\beta_{j};G)-\ol N_{0}(r,0;F^{'})-\ol N_{0}(r,0;G^{'})\\&&+S(r,F)+S(r,G).\eeas

Now using Lemma \ref{2.1}, Lemma \ref{2.2}, Lemma \ref{2.3} and Lemma \ref{2.5}, we have \beas &&\left(m+\frac{1}{2}\right)\{T(r,F)+T(r,G)\}\\&\leq&\;N_{2}(r,0;F)+ N_{2}(r,0;G)+\sum_{j=1}^{m}N_{2}(r,\beta_{j};F)+\sum_{j=1}^{m}N_{2}(r,\beta_{j};G)+\ol N(r,\infty;f)\\&&+\ol N(r,\infty;L(z,f))+\ol N_{*}(r,\infty;F,G_{1})-\left(t-\frac{3}{2}\right)\ol N_{*}(r,1;F,G)+S(r,F)\\&&+S(r,G).\eeas

 i.e.,\beas &&n\left(m+\frac{1}{2}\right)\{T(r,f)+T(r,L(z,f))\}\\&\leq&\; 2\{\ol N(r,0;f)+\ol N(r,0;L(z,f))\}+\ol N(r,\infty;f)+\ol N(r,\infty;L(z,f))\\&&+\sum_{j=1}^{m}N_{2}(r,\beta_{j};F)+\sum_{j=1}^{m}N_{2}(r,\beta_{j};G)+\ol N_{*}(r,\infty;f,L(z,f))\\&&-(t-\frac{3}{2})\ol N_{*}(r,1;F,G)+S(r,f)+S(r,L(z,f)).\eeas
\end{proof}
\begin{lem}\label{2.10}
Let $F,$ $G$ be given by $(\ref{e2.1})$ and $(\ref{e2.2})$, $n\geq8$ is an integer and $V \not\equiv 0$. If $F,$ $G$ share $(1,2)$, and $f,$ $L(z,f)$ share $(\infty,k),$ where $0\leq k<\infty,$ then the poles of $F$ and $G$ are zeros of $V$ and \beas &&(nk+n-1)\ol N(r,\infty;f \mid\geq k+1)\\&=&(nk+n-1)\ol N(r,\infty;L(z,f))\mid\geq k+1)\\&\leq& \ol N(r,0;f)+\ol N(r,0;L(z,f))+\ol N_{*}(r,1;F,G)+S(r,f)+S(r,L(z,f)).\eeas
\end{lem}
\begin{proof}
Since $f(z),$ $L(z,f)$ share $(\infty;k)$, it follows that $F,$ $G$ share $(\infty;nk)$ and so a pole of $F$ with multiplicity $p(\geq nk + 1)$ is a pole of $G$ with multiplicity $r(\geq nk + 1)$ and vice versa. We note that $F$ and $G$ have no pole of multiplicity $q$ where $nk < q < nk + n$.
Now using the Milloux theorem [\cite{17}, p. 55], we get from the definition of $V$, \beas m(r,V)=S(r,f(z))+S(r,L(z,f)).\eeas

Hence \beas &&(nk+n-1)\ol N(r,\infty;f\mid\geq k+1)\\&=&(nk+n-1)\ol N(r,\infty;L(z,f)\mid\geq k+1)\\&&\leq\; N(r,0;V)\\&&\leq\; T(r,V_{1})+O(1)\\&&\leq\; N(r,\infty;V_{1})+m(r,V)+O(1)\\&&\leq\;\ol N(r,0;F)+\ol N(r,0;G)+\ol N_{*}(r,1;F,G)+S(r,f(z))+S(r,L(z,f))\\&&\leq\; \ol N(r,0;f)+\ol N(r,0;L(z,f))+\ol N_{*}(r,1;F,G_{1})+S(r,f)+S(r,L(z,f)).\eeas
\end{proof}
\begin{lem}\label{2.11}\cite{21}
Let $m(\geq 1)$ and $n(> 2m)$ be two positive integers, A, B are non zero complex numbers such that $\displaystyle\frac{A}{B}=\displaystyle\frac{(n-m)^{2}}{n(n-2m)}$. Then the polynomial $$\Phi(h)=A(h^{n}-1)(h^{n-2m}-1)-B(h^{n-m}-1)^{2}$$ of degree $2n-2m$ has one zero of multiplicity $4$ which is 1 and all other zeros are simple.
\end{lem}
\begin{lem}\label{2.12}\cite{11}
 Let $F$, $G$ share $(\infty,0)$ and $V\equiv 0$. Then $F\equiv G$.
\end{lem}

\section{Proof of the theorems}
\begin{proof} [Proof of Theorem \ref{t1}]

Let $F$ and $G$ be two functions defined in (\ref{e2.1}) and (\ref{e2.2}).\par  Since $E_{f}(S,3)=E_{L(z,f)}(S,3)$ and $E_{f}(\{\infty\},0)=E_{L(z,f)}(\{\infty\},0),$ it follows that $F$, $G$ share $(1,3)$ and $(\infty,0).$\par

From (\ref{e1.1f}), we have \beas T(z)-\beta_{i}=-a\displaystyle\frac{\displaystyle\prod_{j=1}^{m}(z-\omega_{j})^{3}R_{n-3m}(z)}{bz^{2m}+cz^{m}+d},\eeas where $R_{n-3m}(z)$ is a polynomial of degree $n-3m$ and $i=1,2,\ldots,m$.\par

Therefore, we have \bea\label{e3.1} N_{2}\Big(r,\beta_{i};F\Big)&\leq& 2\ol N(r,\omega_{j};f)+N(r,0;R_{n-3m}(f))\\&\leq&\nonumber 2\ol N(r,\omega_{j};f)+(n-3m)T(r,f)+S(r,f)\\&\leq& \nonumber (n-m)T(r,f)+S(r,f).\eea

Similarly, \bea \label{e3.2} N_{2}\Big(r,\beta_{i};G\Big)&\leq& (n-m)T(r,L(z,f))+S(r,L(z,f)),\eea for each $i=1,2,\ldots,m.$ \par
\noindent {\bf{Case 1:}} Suppose $H\not\equiv 0$. Then $F\not\equiv G.$ So, it follows from Lemma \ref{2.12} that $V\not\equiv 0.$\par
Using $(\ref{e3.1})$, $(\ref{e3.2})$ and Lemma \ref{2.10} in Lemma \ref{2.8}, we obtain \beas &&n\left(m+\frac{1}{2}\right)\{T(r,f(z))+T(r,L(z,f))\}\\&\leq&\; 2\{\ol N(r,0;f()+\ol N(r,0;L(z,f))\}+\ol N(r,\infty;f)+\ol N(r,\infty;L(z,f))\\&&+\sum_{i=1}^{m}N_{2}(r,\beta_{i};F)+\sum_{i=}^{m}N_{2}(r,\beta_{i};G)+\ol N_{*}(r,\infty;f,L(z,f))\\&&-\left(t-\frac{3}{2}\right)\ol N_{*}(r,1;F,G)+S(r,f)+S(r,L(z,f))\\&\leq&   2\{\ol N(r,0;f)+\ol N(r,0;f,L(z,f))\}+\frac{2}{n-1}\{\ol N(r,0;f)+\ol N(r,0;f,L(z,f))\}\\&& +\frac{1}{nk+n-1}\{\ol N(r,0;f)+\ol N(r,0;f,L(z,f))\}+\sum_{i=1}^{m}\ol N_{2}(r,\beta_{i}; F)\\&& +\sum_{i=1}^{m}\ol N_{2}(r,\beta_{i}; G)-\left(t-\frac{3}{2}-\frac{2}{n-1}-\frac{1}{nk+n-1}\right)\ol N_{*}(r,1;F,G)\\&& +S(r,f)+S(r,L(z,f))\\&\leq& 2\{T(r,f)+T(r,L(z,f))\}+m(n-m)\{T(r,f)+T(r,L(z,f))\}\\&&+\frac{2}{n-1}\{T(r,f)+T(r,L(z,f))\}+\frac{1}{nk+n-1}\{T(r,f)+T(r,L(z,f))\}\\&&-\left(t-\frac{3}{2}-\frac{2}{n-1}-\frac{1}{nk+n-1}\right)\ol N_{*}(r,1;F,G)+S(r,f)+S(r,L(z,f)).\eeas Therefore,

 \bea \label{e3.2a}&& \left(\frac{n}{2}-2-\frac{2}{n-1}-\frac{1}{nk+n-1}+m^{2}\right)\{T(r,f)+T(r,L(z,f))\}\\&\leq&\nonumber -\left(t-\frac{3}{2}-\frac{2}{n-1}-\frac{1}{nk+n-1}\right)\ol N_{*}(r,1;F,G)+S(r,f)+S(r,L(z,f)).\eea

{\bf{Subcase 1.1:}} Suppose $n=2m+3$.\par
{\bf{Subcase 1.1.1:}} Let $m=1$. Then $n=5$. Therefore, putting the values of $m$, $n$, $t=3$ and $k=0$ in $(\ref{e3.2a})$, we get
\beas T(r,f)+T(r,L(z,f))\leq -\ol N(r,1;F,G)+S(r,f)+S(r,L(z,f)).\eeas $i.e.,$ \beas T(r,f)+T(r,L(z,f))\leq S(r,f)+S(r,L(z,f)),\eeas which is a contradiction.

{\bf{Subcase 1.1.2:}} Suppose $m\geq 2$. Then $n\geq 7$. Then putting $t=3$ and $k=0$ in $(\ref{e3.2a})$, we get \beas && \left\{
 \frac{n}{2}-2-\frac{3}{n-1}+m^{2}\right\}(T(r,f)+T(r,L(z,f)))\\&\leq& -\left(\frac{3}{2}-\frac{3}{n-1}\right)\ol N_{\star}(r,1;f,L(z,f))+S(r,f)+S(r,L(z,f)).\eeas
 Since $\displaystyle\frac{3}{2}-\displaystyle\frac{3}{n-1}>0$ for $n\geq 7$, the above equation can be written as \beas \left\{
 \frac{n}{2}-2-\frac{3}{n-1}+m^{2}\right\}(T(r,f)+T(r,L(z,f)))\leq S(r,f)+S(r,L(z,f)),\eeas which is not possible since $m\geq 2$ and $n\geq 7$.

  {\bf{Case 2:}} Suppose $H\equiv 0$. After integration we get, \be\label{e3.3} F\equiv \frac{AG+B}{CG+D},\ee where $A,B,C,D$ are complex constants satisfying $AD-BC\neq 0.$\par As $F$, $G$ share $(\infty,0)$, it follows from (\ref{e3.3}) that $f$, $L(z,f)$ share $(\infty,\infty)$ and $T(r,f(z))=T(r,L(z,f))+S(r,f)$.\par
{\bf{Subcase 2.1:}} Let $AC \neq 0$. Then $F-\displaystyle\frac{A}{C}=\frac{-(AD-BC)}{C(CG+D)}\neq 0.$ \par

Therefore, by the Second Fundamental Theorem of Nevallina, we get \beas nT(r,f)&\leq& \ol N(r,0;F)+\ol N(r,\infty;F)+\ol N(r,\frac{A}{C};F)+S(r,F)\\ &\leq & 2T(r,f)+S(r,f)\eeas which is a contradiction since $n\geq 2m+3$.\par
{\bf{Subcase 2.2:}} Suppose that $AC=0$. Since $AD-BC\neq 0,$ both $A$ and $C$ are not zero simultaneously.\par
{\bf{Subcase 2.2.1:}} Suppose $ A\neq 0$ and $C=0.$ Then (\ref{e3.3}) becomes $ F \equiv \alpha G+\beta $, where $ \alpha=\displaystyle\frac{A}{D}$ and $\beta=\displaystyle\frac{B}{D}.$\par
{\bf{Subcase 2.2.1.1:}} Let $F$ has no $1$-point. Then by the Second Fundamental Theorem, we get \beas T(r,F)\leq \ol N(r,0;F)+\ol N(r,1;F)+\ol N(r,\infty;F)+S(r,F)\eeas $or,$ \beas (n-2)T(r,f)\leq S(r,f),\eeas which is a contradiction.\par
{\bf{Subcase 2.2.1.2:}} Let $F$ has some $1$-point. Then $\alpha+\beta=1$.\par
{\bf{Subcase 2.2.1.2.1:}} Suppose $\alpha\neq1$. Then $ F \equiv \alpha G+1-\alpha$.\par
 Therefore, by the Second Fundamental Theorem, we get \beas &&(m+1)T(r,F)\\&\leq& \ol N(r,0;F)+\ol N(r,\infty;F)+\ol N(r,1-\alpha;F)+\sum_{j=1}^{m}\ol N(r,\beta_{j};F)+S(r,F)\\&\leq& \ol N(r,0;F)+\ol N(r,\infty;F)+\ol N(r,0;G)+\sum_{j=0}^{m}\ol N(r,\beta_{j};F)+S(r,F)\\&\leq& 3T(r,f)+mnT(r,f).\eeas $i.e.,$ \beas (n-3)T(r,f)\leq S(r,f),\eeas which is again a contradiction since $n\geq 2m+3$.\par


\par {\bf{Subcase 2.2.1.2.2:}} Suppose $\alpha=1$. Then $F\equiv G.$\par
$i.e.,$ \beas && \frac{af(z)^{n}}{bf(z)^{n}+cf(z)^{m}+d}\equiv \frac{aL(z,f)^{n}}{bL(z,f)^{2m}+cL(z,f)^{m}+d}.\eeas

$i.e.,$ \beas && f(z)^{n}\left(bL(z,f)^{2m}+cL(z,f)^{m}+d\right)\equiv L(z,f)^{n}\left(bf(z)^{n}+cf(z)^{m}+d\right).\eeas Suppose that $h(z)=\frac{L(z,f)}{f(z)}.$
Then the above equation can be written as
\bea \label{e3.5}d(h^{n}-1)+cf^{m}h^{m}(h^{n-m}-1)+bf^{2m}h^{2m}(h^{n-2m})=0.\eea

 Suppose $h(z)$ is not constant.\par
 After some simple calculation the above equation becomes
  \beas \left(bf^{m}h^{m}(h^{n-2m})+\frac{c}{2}(h^{n-m}-1)\right)^{2}&=&\frac{\left( c^{2}(h^{n-m}-1)^{2}-4bd(h^{n-2m}-1)(h^{n}-1)\right)}{4}\\&=& \frac{1}{4}\Phi(h).\eeas
Therefore in view of Lemma \ref{2.11},  the above equation takes the form \beas \left(bf^{m}h^{m}(h^{n-2m})+\frac{c}{2}(h^{n-m}-1)\right)^{2}=\frac{1}{4}(h-1)^{4}\prod_{j=1}^{2n-2m-4}(h-\eta_{j}),\eeas where  $\eta_{1},\eta_{2},\ldots, \eta_{2n-2m-4}$ are the simple zeros of $\Phi(h).$\par

From the above equation it is clear that all the zeros of $h-\eta_{j}$ have order atleast $2$.\par
Therefore, by the Second Fundamental Theorem, we have \beas (2n-2m-1)T(r,h)&\leq & \sum_{j=1}^{2n-2m-4}\ol N(r,\eta_{j};h)+S(r,h)\\& \leq &\frac{1}{2}\sum_{j=1}^{2n-2m-4} N(r,\eta_{j};h)+S(r,h)\\& \leq & (n-m-2)T(r,h)+S(r,h).\eeas
$i.e.$,
$$(n-m-3)T(r,h)\leq S(r,h),$$ which is impossible since $n\geq 2m+3$.
\par
 So, $h$ is constant. Hence, from $(\ref{e3.5})$, we have $h^{n}-1=0$, $h^{n-m}-1=0$ and $h^{n-2m}-1=0$. Since $gcd(n,m)=1$, we must have $h\equiv 1$.\par $i.e.$, $$f(z)\equiv L(z,f).$$
{\bf{Subcase 2.2.2:}} Suppose $A=0$ and $C\neq 0$.\par

Then (\ref{e3.3}) becomes \beas F\equiv \frac{1}{\gamma G+\delta},\eeas  where $\gamma=\frac{C}{B}$ and $\delta=\frac{D}{B}.$\par
{\bf{Subcase 2.2.2.1:}} Let $F$ has no $1$-point. Then applying the second fundamental theorem to $F$, we have \beas nT(r,f)&\leq &\ol N(r,\infty;F)+\ol N(r,0;F)+\ol N(r,1;F)+S(r,F)\\&\leq & 2T(r,f)+ S(r,f),\eeas which is a contradiction.\par
{\bf{Subcase 2.2.2.2:}} Suppose that $F$ has some $1$-point. Then $\gamma+\delta=1$.\par
{\bf{Subcase 2.2.2.2.1:}} Suppose $\gamma=1$. Then $\delta=0$ and thus $FG\equiv 1$.\par $i.e.$, \bea\label{e3.6} f(z)^{n}L(z,f)^{n}\equiv (bf^{2m}+cf^{m}+d)(bL(z,f)^{2m}+cL(z,f)^{m}+d).\eea
Since $c^{2}\neq 4bd$, by simple calculation it can be easily seen that all the roots of the equation $bz^{2m}+cz^{m}+d=0$ are simple. Let them be $\delta_{1},\delta_{2},\ldots, \delta_{2m}$.\par
Therefore, $(\ref{e3.6})$ can be written as \bea\label{e3.7} f(z)^{n}L(z,f)^{n}=\prod_{j=1}^{2m}(f-\delta_{j})\prod_{j=1}^{2m}(L(z,f)-\delta_{j})\eea
From $(\ref{e3.7})$, it is clear that the order of each $\delta_{j}$ points of $f(z)$ is atleast $n$.
\par
As $f(z)$ and $L(z,f)$ share $(\infty,\infty)$, from $(\ref{e3.6})$ it is to be observed that $\infty$ is a e.v.p. of both $f(z)$ and $L(z,f)$.
Therefore, applying the second fundamental theorem of Nevallina to $f$, we get \beas (2m-1)T(r,f)\leq \frac{2m}{n}T(r,f),\eeas which is a contradiction for $n\geq 2m+3$.

{\bf{Subcase 2.2.2.2.2:}} Let $\gamma\neq 1$.
Therefore, $$ F\equiv \frac{1}{\gamma G+1-\gamma}.$$ Since $C\neq 0,$ $\gamma\neq 0$, $G$ omits the value $-\frac{1-\gamma}{\gamma}.$\par
By the Second Fundamental Theorem of Nevalinna, we have \beas T(r,G)&\leq &\ol N(r,\infty;G)+\ol N(r,0;G)+\ol N(r,-\frac{1-\gamma}{\gamma};G)+S(r,G).\eeas $i.e.,$ \beas (n-2)T(r,L(z,f))\leq S(r,L(z,f)),\eeas which is a contradiction.\par This completes the proof of the theorem.\end{proof}

\begin{proof} [Proof of Corollary \ref{c1}]
	The proof of the corollary can be carried out in the line of the proof of theorem \ref{t1}.
	\end{proof}



\begin{thebibliography}{99}
\bibitem{Aha & RM & 2019} Ahamed, M.B., \emph{An investigation on the Conjecture of Chen and Yi}, Results Math., 74: 122 (2019).
\bibitem{Aha & BB & 2019} Ahamed, M.B., \emph{On the periodicity of meromorphic functions sharing two sets IM}, Stud. Univ. Babes-Bolyai Math. 64(3)(2019), 497-510.
\bibitem{2}  Banerjee, A., Uniqueness of meromorphic functions sharing two sets with finite weight II, Tamkang J. Math., 41(4)(2010), 379-392.
\bibitem{2*}  Banerjee, A., Ahamed, M.B., On Some sufficient conditions for periodicity of meromorphic function under new shared sets, FILOMAT 33(18)(2019), 6055-6072.
\bibitem{Ban & Aha & MS & 2019} Banerjee, A., Ahamed,M.B., \emph{Uniqueness of meromorphic function with its shift operator under the perview of two or three shared sets}, Math. Slovaca 69(3) (2019), 557-572.
\bibitem{Ban & Aha & JCMA & 2020} Banerjee, A., Ahamed, M.B., \emph{Results on meromorphic function sharing two sets with its linear $ c $-shift operator}, J. Contemp. Math. Anal., 55(3)(2020),143-155.
\bibitem{21} Banerjee, A., Mallik, S., On the characteristics of a new class of strong uniqueness polynomials generating unique range sets, Comput. Methods Funct. Theo., 17(2017), 19-45.
\bibitem{29} Chen, B., Chen, Z., Meromorphic functions sharing two sets with its difference operator, Bull. Malyays. Math. sci. Soc., (2)35(3)(2012), 765-774.
\bibitem{23} Chiang, Y. M., Feng, S. J., On the Nevalinna characteristic of $f(z+\eta)$ and difference equations in the complex plane, Ramanujan J., 16(2008), 105-129.
\bibitem{Hal & Kor & JMAA & 2006} Halburd, R.G., Korhonen, R.J., \textit{Difference analogue of the lemma on the logarithmic derivative with applications to difference equations}, J. Math. Anal. Appl., 314(2006), 477-487.
\bibitem{17}  Hayman, W. K., Meromorphic Functions, The Clarendon Press, Oxford (1964).
\bibitem{Hei & Kor & Lai & Rei & JMMA & 2009} Heittokangas,  J., Korhonen, R.J., Laine, I., Rieppo, J., Zhang, J.L., \textit{Value sharing results for shifts of meromorphic function, and sufficient conditions for periodicity}, J. Math. Anal. Appl., 355(2009), 352–363.
\bibitem{12a}  Lahiri, I., Value distribution of certain differential polynomials, Int. J. Math. Math. Sci., 28(2)(2001), 83-91.
\bibitem{13}  Lahiri, I., Weighted sharing and uniqueness of meromorphic functions, Nagoya Math. J., 161(2001), 193-206.
\bibitem{14}  Lahiri, I., Weighted value sharing and uniqueness of meromorphic functions, Complex Var. Theory Appl., 46(2001), 241-253.
\bibitem{18} Li, P., Yang, C. C., Some further results on the unique range sets of meromorphic functions, Kodai Math. J., 18(1995), 437-450.
\bibitem{18a} Li, S., Chen, B, Unicity of meromorphic functions sharing sets with their linear difference polynomials, Publishing Corporation, Abstract and Applied Analysis, Volume 2014, Article ID 894968, 7 pages.
\bibitem{8} Yang, C. C., On deficiencies of differential polynomials II, Math. Z., 125(1972), 107-112.
\bibitem{11} Yi,  H. X., Yang, L. Z., Meromorphic functions that share two sets, Kodai Math. J., 20(1997), 127-134.
\bibitem{16} Zhang, J., Value distribution and shared sets of difference of meromorphic functions, J. Math. Anal. Appl., 367(2010), 401-408.


\end{thebibliography}
\end{document}